\documentclass{amsart}
              [1997/02/02 v1.2e PROC Author Class]

\copyrightinfo{2006}
  {American Mathematical Society}

\newtheorem{theorem}{Theorem}[section]
\newtheorem{corollary}[theorem]{Corollary}

\newtheorem{proposition}[theorem]{Proposition}

\begin{document}

\title[Separately continuous functions of many variables on products]{Separately continuous functions of many variables on product of spaces which are products of metrizable multipliers}

\author{V.K.Maslyuchenko}
\address{Department of Mathematics\\
Chernivtsi National University\\ str. Kotsjubyn'skogo 2,
Chernivtsi, 58012 Ukraine}
\email{mathan@chnu.cv.ua}

\author{V.V.Mykhaylyuk}
\address{Department of Mathematics\\
Chernivtsi National University\\ str. Kotsjubyn'skogo 2,
Chernivtsi, 58012 Ukraine}
\email{vmykhaylyuk@ukr.net}



\commby{Ronald A. Fintushel}



\begin{abstract}
It is obtained necessary and sufficient conditions of dependence
on $\aleph$ coordinates for functions of several variables,
each of which is a product of metrizable factors. The set of discontinuity
points of such functions is characterised in the case, when each
variable is a product of separable metrizable spaces.
\end{abstract}

\maketitle
\section{Introduction}

Dependence on $\aleph$ coordinates of separately continuous functions defined on the product of two spaces which are products of compacts was investigated in paper [1] using a theorem on density of topological product and a theorem of dependence on at most countable coordinates of continuous function defined on the product of compacts.
In particular, it was obtained that if $X$ and $Y$ are the products of metrizable compacts then every separately continuous function $f:X\times Y \to \bf R$ depends on at most countable coordinates. This gives the possibility to characterize the discontinuity points set of such functions. It was establishes in [2] that the dependence on some coordinates of separately continuous functions of two variables is closely connected with product properties which defined using families of open sets. A similar connection for continuous mappings was obtained in [3]. In this paper we consider separately continuous functions of many variables. Developing results from [2], we obtain necessary and sufficient conditions for the dependence on certain number of coordinates which coincides in the case when every variable is the product of metrizable spaces.  Using this result we characterize the discontinuity points set for separately continuous functions many variables in the case when every variable is the product of separable metrizable spaces.

\section{Definitions and auxiliary statements}

Let $X=\prod\limits _{s\in S}X_s$ be the product of a family of sets $X_s$, $Z$ be a set and $T\subseteq S$. We say that {\it a mapping $f:X\to Z$ concentrated on $T$}, if $f(x')=f(x'')$ for $x',x''\in X$ with $x'_{|_T}=x''_{|_T}$. Moreover, if $|T|\le\aleph$, then we say that  {\it $f$ depends at most on $\aleph$ coordinates}. Let $Y$ be a set. We say that {\it a mapping $g:X\times Y \to Z$ concentrated on $T$ with respect to the first variable}, if the mapping $\varphi :X\to Z^Y$, $\varphi (x)(y)=g(x,y)$, concentrated on $T$. Moreover, if $|T|\le\aleph$, then we say that {\it $g$ depends at most on $\aleph$ coordinates with respect to the first variable}.

Let $P=X_1\times\cdots\times X_n$, $X_i=\prod\limits _{s\in S_i} X_{i,s}$, $f:P\to Z$ be a mapping. Fix an index $i=1,...,n$ and denote by $\tilde f$ the mapping $\tilde f:X_i\times\hat X_i\to Z$, $\tilde f(x_i,\hat x_i)=f(x_1,...,x_n)$, where $\hat X_i=X_1\times \cdots \times X_{i-1}
\times X_{i+1}\times \cdots X_n$ and $\hat x_i=(x_1,\cdots ,x_{i-1},x_{i+1}, \cdots ,x_n)$. We say that {\it $f$ concentrated on $T\subseteq S_i$ with respect to $i$-th variable}, if $\tilde f$ concentrated on $T$ with respect to the first variable. Moreover, if $|T|\le\aleph$, then we say that {\it $f$ depends on at most $\aleph$ coordinates with respect to $i$-th variable}. For an abridgment we shall use the term "depends on $\aleph$ coordinates" instead the term "depends on at most $\aleph$ coordinates". If we denote by $S$ the direct sum of sets $S_1,...,S_n$ and put $Y_s=X_{i,s}$ for $s\in S_i$ and $Y=\prod\limits _{s\in S} Y_s$, then $P$ can be identified with $Y$ and $f$ is a mapping defined on $Y$ with values in $Z$. Note that if a cardinal $\aleph$ is infinite then $f:Y\to Z$ depends on $\aleph$ coordinates if and only if $f$ depends on $\aleph$ coordinates with respect to $i$-th variable for every $i=1,...,n$.

Let $X$ be a topological space and $\aleph$ be an infinite cardinal. We say that a family $\alpha=(A_i:i\in I)$ of sets $A_i\subseteq X$ is {\it locally finite}, if for every $x\in X$ there exists a neighborhood $U$ of $x$ such that the set $\{i\in I:A_i\cap U\ne \O \}$ is finite, {\it poinwise finite} ($\aleph$-{\it pointwise}), if for every $x\in X$ the set $\{i\in I: x\in A_i\}$ is finite (has the cardinality $\leq\aleph$). Moreover, for arbitrary family $\alpha=(A_i:i\in I)$ the cardinality of $I$ we shall call by the cardinality of the family $\alpha$.

The following properties of a topological space $X$ will be useful for our investigation:

(I$_\aleph$) every locally finite family of open nonempty subsets of $X$ has the cardinality $\leq\aleph$;

(II$_\aleph$) every poinwise finite family of open nonempty subsets of $X$ has the cardinality $\leq\aleph$;

(III$_\aleph$) every $\aleph$-pointwise family of open nonempty subsets of $X$ has the cardinality $\leq\aleph$.

Clearly that (III$_\aleph$)$\Longrightarrow $(II$_\aleph$)$\Longrightarrow $(I$_\aleph$).

We shall use the following result [3, Proposition 1]: a topological product
$X=\prod\limits _{s\in S} X_s$ has property (I$_\aleph$), (II$_\aleph$) or (III$_\aleph$) if and only if the same property has the product
$X(T)=\prod\limits _{s\in T} X_s$ for every finite set $T\subseteq S$.

This result on the property (III$_\aleph$) can be strengthened in such a way.

\begin{proposition}\label{p:2.1} The product $X\times Y $ of topological spaces $X$ and $Y$ has the property (III$_\aleph$) if and only if the spaces $X$ and $Y$ have (III$_\aleph$).
\end{proposition}

\begin{corollary}\label{cor:2.2}
The topological product of a family of topological spaces has the property (III$_\aleph$) if and only if every multiplier has this property.
\end{corollary}

By $d(X)$ we denote {\it the density} of the topological space $X$.

\begin{proposition}\label{p:2.3}
Let $X$ be a topological space and $\aleph$ be an infinite cardinal. Then:

(i) if $d(X)\le\aleph$, then $X$ has (III$_\aleph$);

(ii) if $X$  is metrizable, then all properties (I$_\aleph$), (II$_\aleph$),
(III$_\aleph$) are equivalent to $d(X)\le\aleph$.
\end{proposition}

\begin{proof} $(i)$. Let $\alpha=(U_i:i\in I)$ be an $\aleph$-pointwise family of nonempty open in $X$ sets and $A$ is dense in $X$ set with $|A|\le\aleph$. For $x\in A$ we consider the set $I(x)=\{i\in I:x\in U_i\}$ for which $|I(x)|\le\aleph$. Since $U_i$ are open and nonempty and $\overline{A}=X$, $I=\bigcup \limits _{x\in A}I(x)$. Therefore $|I|\le\aleph^2=\aleph$.

$(ii)$. Let $X$ be a metrizable space with (I$_\aleph$). It is wellknown [4,p.416] that there exists in $X$ an $\sigma $-locally finite base ${\mathcal B}$ for which $|{\mathcal B}|\le\aleph$. Then $d(X)\le\aleph$.
\end{proof}

Note that there exists a topological space $X$ with $d(X)>\aleph$ and which has (III$_\aleph$). For example, $X=[0,1]^S$ where $|S|> 2^{\aleph}$.

\section{Necessary conditions of dependence}

We shall use the following result [2, Corollary 1].

\begin{proposition}\label{p:3.1}
Let $X=\prod\limits _{s\in S} X_s$ be the topological product of a family  of topological spaces $X_s$, $Y$ be a set, $Z$ be a Hausdorff space and $f:X\times Y\to Z $ be a function which is continuous with respect to the first variable. Then the set
$$
S_0=\{s\in S:(\exists y\in Y)(\exists u,v\in X)(u_{|_{S\setminus
\{s\}}}=v_{|_{S\setminus \{s\}}}\,\,\,\mbox{and}\,\,\, f(u,y)\ne
f(v,y))\}
$$
is a smallest set among all sets on which $f$ concentrated with respect to the first variable.
\end{proposition}

Further we assume that every topological space $X_s$ and $Y_t$ contains at least two points.

The following result generalizes the Theorem 2 from [2] and gives necessary conditions for the dependence on $\aleph$ coordinates.

\begin{theorem}\label{th:3.2} Let $\aleph$ be an infinite cardinal, $|S|>\aleph$,
$X=\prod\limits _{s\in S}X_s$ be the topological product of a family of completely regular spaces $X_s$, $Y_1,...,Y_n$ be completely regular spaces and every separately continuous function $f:X\times Y_1\times...\times Y_n\to {\bf R} $ depends on $\aleph$ coordinates with respect to the first variable. Then the product $X\times Y_1\times...\times Y_n$ has $(I_\aleph)$ and all spaces $X,Y_1,...,Y_n$, except perhaps one have $(II_\aleph)$.
\end{theorem}

\begin{proof} Suppose that there are two spaces among $X,Y_1,...,Y_n$ which have not (II$_\aleph$). Then there exists $k\leq n$ such that $Y_k$ has not  (II$_\aleph$). Without limiting the generality we assume that $k=n$. Choose some set $\tilde T=\{t_1,...,t_{n-1}\}$ with $|\tilde T|=n-1$ such that $S\cap \tilde T=\emptyset$. Put $T=S\cup\tilde T$, $Y=Y_n$, $Z_s=X_s$ for every $s\in S$, $Z_{t_k}=Y_k$ for $k=1,...,n-1$ and $Z=\prod\limits _{t\in T}Z_t$. Clearly that $Z$ and $Y$ have not (II$_\aleph$) and are completely regular. Then it follows from [2, Theorem
2] that there exists a separately continuous function $g:Z\times Y\to{\bf R} $ which depends not on $\aleph$ coordinates with respect to the first variable. According to Proposition \ref{p:3.1} the set
$$
  T_0=\{t\in T:(\exists y\in Y)(\exists u,v\in Z)(u|_{T\setminus \{t\}}=v|_{T\setminus \{t\}} \,\, {\rm and} \,\, g(u,y) \ne g(v,y))\}
$$
has the cardinality $>\aleph$. Consider the function $f:X\times Y_1\times...\times Y_n\to {\bf R} $ for which
$f(x,y_1,...,y_n)=g(z,y)$ where $z=(x,y_1,...,y_{n-1})$ and $y=y_n$. The separately continuity of $g$ implies the separately continuity of $f$. Moreover, it easy to see that the set $S_0=T_0\cap S$ is a smallest set among all sets on which $f$ concentrated with respect to the first variable. Since $|T_0|>\aleph$ and $\tilde T$ is finite, $|S_0|=|T_0\setminus \tilde T|>\aleph$. Thus, $f$ depends not on $\aleph$ coordinates with respect to the first variable, a contradiction.
Similar arguments show that the product  $X\times Y_1\times...\times Y_n$ has (I$_\aleph$).
\end{proof}

\section{Sufficient conditions of dependence}

\begin{theorem}\label{th:4.1} Let $\aleph$ be an infinite cardinal, a topological product $X=\prod\limits _{s\in S}X_s$ has (I$_\aleph$), spaces $Y_1,...,Y_n$
have (III$_\aleph$). Then every separately continuous function
\mbox{$f:X\times Y_1\times...\times Y_n\to {\bf R} $} depends on $\aleph$ coordinates with respec to the first coordinates.
\end{theorem}

\begin{proof} Let  $Y=Y_1\times...\times Y_n$. We consider the set $S_0$ from Proposition \ref{p:3.1}. According to Proposition \ref{p:3.1}, it is sufficient to prove that $|S_0|\le \aleph$. Suppose the contrary, that is $|S_0|>\aleph$. For every $s\in S_0$ we choose $y_s=(y_{s,1},...,y_{s,n})\in Y$ and $u_s, v_s\in X$ such that $u_{s|_{S\setminus \{s\}}}=v_{s|_{S\setminus \{s\}}}$ and $f(u_s,y_s) \ne f(v_s,y_s)$. For every $k\in \bf N $ we put
$$
  S_k=\{s\in S_0:|f(u_s,y_s)-f(v_s,y_s)|>1/k\}.
$$
Clearly that $S_0=\bigcup \limits _{k=1}^\infty S_k$. Since $|S_0|>\aleph$, there exists an integer $k_0$ such that $|S_{k_0}|>\aleph$. Put $T_0=S_{k_0}$ and $\varepsilon =1/k_0$. It follows from the continuity of $f$ with respect to the second variable that for every $s\in T_0$ there exists an open neighborhood $V_{1,s}$ of $y_{1,s}$ in $Y_1$ such that
$$
  |f(u_s,y_1,y_{2,s},...,y_{n,s})-f(v_s,y_1,y_{2,s},...,y_{n,s})|>\varepsilon ,
$$
for $y_1\in V_{1,s}$. The family $(V_{1,s}:s\in T_0)$ is not $\aleph$-pointwise in $Y_1$, because $|T_0|>\aleph$ and $Y_1$ has (III$_\aleph$). Therefore there exists $b_1\in Y_1$ such that the set $T_1=\{s\in T_0:b_1\in V_{1,s}\}$ has the cardinality $>\aleph$. Note that
$$
  |f(u_s,b_1,y_{2,s},...,y_{n,s})-f(v_s,b_1,y_{2,s},...,y_{n,s})|>\varepsilon
$$
for every $s\in T_1$. It follows from the continuity of $f$ with respect to the third variable that for every $s\in T_1$ there exists an open neighborhood $V_{2,s}$ of $y_{2,s}$ in $Y_2$ such that
$$
  |f(u_s,b_1,y_{2},y_{3,s}...,y_{n,s})-f(v_s,b_1,y_2,y_{3,s},...,y_{n,s})|>\varepsilon
$$
for $y_2\in V_{2,s}$. The family $(V_{2,s}:s\in T_1)$ is not $\aleph$-pointwise in $Y_2$, because $|T_1|>\aleph$ and $Y_2$ has (III$_\aleph$). Therefore there exists $b_2\in Y_2$ such that the set $T_2=\{s\in T_1:b_2\in V_{2,s}\}$ has the cardinality $>\aleph$. Then
$$
  |f(u_s,b_1,b_{2},y_{3,s}...,y_{n,s})-f(v_s,b_1,b_2,y_{3,s},...,y_{n,s})|>\varepsilon
$$
for every $s\in T_2$.

Making $n-2$ such steps we obtain a set $T_n\subseteq S$ and a point $b=(b_1,...,b_n)\in Y$ such that $|T_n|>\aleph$ and $|f(u_s,b)-f(v_s,b)|>\varepsilon $ for every $s\in T_n$. This implies that the function $f_b:X\to{\bf R}$, $f_b(x)=f(x,b)$, essentially depends on every coordinate $x_s$ with $s\in T_n$. But the function $f_b$ is continuous and $X$ has (I$_\aleph$). Therefore according to Noble-Ulmer Theorem [3, Theorem 3.2] $f_b$ depends on $\aleph$ coordinates, a contradiction.
\end{proof}

\begin{corollary}\label{cor:4.2}
Let all topological products $X_i=\prod\limits _{s\in S_i} X_{i,s}$ for $i=1,...,n$ have (III$_\aleph$). Then every separately continuous function
$f:X_1\times\cdots\times X_n\to{\bf R} $ depends on $\aleph$ coordinates.
\end{corollary}

If all spaces $X_{i,s}$ are metrizable we obtain necessary and sufficient conditions for dependence on $\aleph$ coordinates of separately continuous function.

\begin{theorem}\label{th:4.3}
Let $\aleph$ be an infinite cardinal, $X_i=\prod\limits _{s\in S_i} X_{i,s}$ for $i=1,...,n$ be topological products of families of metrizable spaces $X_{i,s}$,  $|S_1|>\aleph$ and $X=X_1\times\cdots\times X_n$. Then the following conditions are equivalent:

(i) every separately continuous function $f:X_1\times\cdots\times X_n\to{\bf R}$ depends on $\aleph$ coordinates;

(ii) every separately continuous functions $f:X_1\times\cdots\times X_n\to{\bf R} $ depends on $\aleph$ coordinates with respect the first variables;

(iii) every continuous function $f:X\to \bf R $ depends on $\aleph$ coordinates;

(iv) $X$ has (I$_\aleph$);

(v) every space $X_{i,s}$ has (I$_\aleph$);

(vi) $X$ has (II$_\aleph$);

(vii) every space $X_{i,s}$ has (II$_\aleph$);

(viii) $X$ has (III$_\aleph$);

(ix) every space $X_{i,s}$ has (III$_\aleph$);

(x) $d(X_{i,s})\le\aleph$  for every $i$  and $s$.
\end{theorem}

\begin{proof}
The implication (i)$\Longrightarrow $(ii) is obvios. The implication (ii)$\Longrightarrow $(iv) follows from Theorem \ref{th:3.2}. According to Noble-Ulmar Theorem [3, Theorem 3.2] the conditions (iii) and (iv) are equivalent. The equivalence of (viii) and (ix) follows from Corollary \ref{cor:2.2}. The equivalence of (v), (vii), (ix) and (x) follows from Proposition \ref{p:2.3}. Note that the product $X$ has one of the properties (I$_\aleph$), (II$_\aleph$) or (III$_\aleph$) if and only if every their finite subproduct has the same property. Moreover, the product of finite number of metrizable space is metrizable. Therefore according to Proposition \ref{p:2.3}, the properties (iv), (vi) and (viii) are equivalent. Thus the properties (iv)--(x) are equivalent. Finally, the implication (viii)$\Longrightarrow $(i) follows from Corollaries \ref{cor:2.2} and \ref{cor:4.2}.
\end{proof}

\section{Main result}

A subset $E$ of the product $X_1\times \cdots \times
X_n$ of topological spaces $X_1, \cdots ,X_n$ is called º {\it projectively meagre}, if for every $i=1,\cdots,n$ the projection of $E$ in parallel to $i$-th multiplier is meagre in  $X_1 \times \cdots \times X_{i-1} \times X_{i+1} \times \cdots \times X_n$.

\begin{theorem}
Let $X_i=\prod\limits _{s\in S_i} X_{i,s}$ for $i=1,...,n$ be topological products of families of metrizable separable spaces $X_{i,s}$,  $X=X_1\times\cdots\times X_n$
and $E\subseteq  X$. Then $E$ is the discontinuiuty points set for a separately continuous function $f:X_1\times\cdots\times X_n\to{\bf R}$ if and only if there exist at most countable sets $T_i\subseteq S_i$ for $i=1,..,n$ and projectively meagre $F_{\sigma}$-set $E_0$ in a product $Y=Y_1\times\cdots\times Y_n$, where $Y_i=\prod\limits _{s\in T_i} X_{i,s}$ such that $E={\rm pr} ^{-1}(E_0)$ where ${\rm pr} :X\to Y $ is the natural projection for which  ${\rm pr} (x_1,...,x_n)=(x_{1|_{T_1}},...,x_{n|_{T_n}})$.
 \end{theorem}

\begin{proof} Let $f:X\to \bf R$ be a separately continuous function and the set $D(f)$ of the discontinuity points set of $f$ coincides with $E$. Since the spaces $X_{i,s}$ are metrizable and $d(X_{i,s})\le\aleph_0$, according to Theorem \ref{th:4.1} the function $f$ depends on $\aleph_0$ coordinates. Then for every $i=1,...,n$ there exists an at most countable set $S_i\subseteq T_i$ such that $f=f_0\circ{\rm pr}$ where $f_0:Y\to{\bf R} $ be a function, ${\rm pr} $ and $Y$ are as in the formulation of the Theorem. The mapping ${\rm pr} :X\to Y $ is continuous and open, therefore $f_0$ is separately continuous and $D(f)={\rm pr} ^{-1}(D(f_0))$. Let $D(f_0)=E_0$. Since the spaces  $Y_1,...,Y_n$ are metrizable and separable, the set $E_0$ is a projectively meagre $F_{\sigma}$-set in $Y$ (see [5, Theorem 2]). Now we have $E={\rm pr} ^{-1}(E_0)$. Thus, the necessity is proved.

For the proof of the sufficiency we note that Theorem 2 from [5] implies that for projectively meagre $F_{\sigma}$-set $E_0$ in $Y$ there exists separately continuous function $f_0:Y\to{\bf R} $ such that $D(f_0)=E_0$. Then the function $f=f_0\circ{\rm pr} :X\to{\bf R} $ is separately continuous and
 $D(f)={\rm pr} ^{-1}(E_0)=E$.
 \end{proof}

\bibliographystyle{amsplain}

\end{document}